\documentclass[a4paper, 11pt,reqno]{amsart}

\usepackage{amsmath}
\usepackage{amssymb}
\usepackage{hyperref} 
\usepackage{amsfonts}
\usepackage{mathtools}
\usepackage{amsthm}
\usepackage{color}
\usepackage[utf8]{inputenc}
\usepackage{array}
\usepackage{enumitem}
\usepackage[english]{babel}

\calclayout

\setlist{topsep=0mm,partopsep=0mm,itemsep=1mm}

\theoremstyle{plain}
\newtheorem{thm}{Theorem}
\newtheorem{impcor}[thm]{Corollary}
\newtheorem{lemma}{Lemma}[section]
\newtheorem{cor}[lemma]{Corollary}

\newtheorem*{thm*}{Theorem}

\theoremstyle{definition}
\newtheorem{que}[lemma]{Question}

\theoremstyle{remark}

\newcommand{\N}{\mathbb{N}}
\newcommand{\Z}{\mathbb{Z}}

\newcommand{\R}{\mathcal{R}}
\renewcommand{\L}{\mathcal{L}}
\renewcommand{\H}{\mathcal{H}}
\newcommand{\J}{\mathcal{J}}

\begin{document}

\title[Number of subsemigroups of direct products]{On the number of subsemigroups of direct products involving the free monogenic semigroup}
\author[A. Clayton, N. Ru\v{s}kuc]{Ashley Clayton \& Nik Ru\v{s}kuc}
\address{School of Mathematics and Statistics, University of St Andrews, St Andrews, Scotland, UK}
\email{$\{$ac323,nr1$\}$@st-andrews.ac.uk}


\begin{abstract}
The direct product $\N\times\N$ of two free monogenic semigroups contains uncountably many pairwise non-isomorphic subdirect products.
Furthermore, the following hold for $\N\times S$, where $S$ is a finite semigroup.
It  contains only countably many pairwise non-isomorphic  subsemigroups if 
and only if $S$ is a union of groups. And it contains only countably many pairwise non-isomorphic subdirect products if and only if every element of $S$ has a relative left- or right identity element.
\end{abstract}

\maketitle

\section{Introduction}
\label{sec1}

It is well known in group theory that the subgroups of the direct product of two groups can exhibit fairly wild behaviour, even if the factors are well-behaved.
For instance, Baumslag and Roseblade showed in \cite[Theorem 1]{baumslaugandroseblade}
that there are uncountably many pairwise non-isomorphic subgroups of the direct product $F_2\times F_2$ of two free groups of rank $2$.
In fact, the subgroups they construct are all subdirect products, meaning that they project onto both copies of $F_2$. Among them, there are subgroups which are:
not finitely generated \cite[Example 3]{bridsonmiller};
finitely generated but not finitely presented \cite{Grunewald78};
 finitely generated but with an undecidable membership problem 
  \cite{Mihailova66}.

On the other hand, the subgroups of the direct product $\Z \times\Z $
of two cyclic groups are much more benign: each non-trivial such subgroup is isomorphic to $\Z$ or $\Z\times\Z$; hence they are all finitely generated, finitely presented, there are only countably many of them, and only finitely many up to isomorphism.
The subsemigroups of the free monogenic semigroup $\N=\{1,2,3,\dots\}$ are also reasonably tame, even though they are more complicated than subgroups of $\Z $.
Every such subsemigroup has the form
$A\cup B$, where $A$ is finite, and $B=\{ nd\::\: n\geq n_0\}$ for some $n_0,d\in\N$;
see \cite{sitandsiu}.
Consequently they are all finitely generated and finitely presented, and there are only countably many of them. 

One might therefore hope that this tame behaviour carries over to the subsemigroups of $\N\times\N$.
This, however, is not the case, and we prove the following:

\begin{thm}\label{thmA}
There are uncountably many pairwise non-isomorphic subsemigroups of $\N \times \N$.
\end{thm}

\begin{impcor}\label{corrC}
If $S$ and $T$ are semigroups containing elements of infinite order, then $S\times T$ contains uncountably many pairwise non-isomorphic subsemigroups.
\end{impcor}

\begin{thm}\label{thmB}
For any $k \geq 2$, the direct power $\N^{k}$ contains uncountably many pairwise non-isomorphic subdirect products.
\end{thm}

We also investigate the subsemigroups of the direct products of the form $\N\times S$, where $S$ is a finite semigroup, and prove that even there we mostly have uncountably many subsemigroups or subdirect products, characterising in the process precisely when the number is countable:

\begin{thm}\label{thmD} The following are equivalent for a finite semigroup $S$:
\begin{enumerate}[label=\textup{(\roman*)}, widest=(iii), leftmargin=10mm]
\item 
\label{it:D1}
$\N \times S$ has only countably many subsemigroups;
\item 
\label{it:D2}
$\N \times S$ has only countably many pairwise non-isomorphic subsemigroups;
\item 
\label{it:D3}
$S$ is a union of groups.
\end{enumerate}
\end{thm}

\begin{thm}\label{thmE} The following are equivalent for a finite semigroup $S$:
\begin{enumerate}[label=\textup{(\roman*)}, widest=(iii), leftmargin=10mm]
\item 
\label{it:E1}
$\N \times S$ has only countably many subdirect products;
\item 
\label{it:E2}
$\N \times S$ has only countably many pairwise non-isomorphic subdirect products;
\item 
\label{it:E3}
For every $s \in S$, there exists some $t \in S$ such that at least one of $ts = s$ or $st = s$ holds.
\end{enumerate}
\end{thm}

We remark that the class of semigroups described in Theorem \ref{thmD} is strictly contained within the class described in Theorem \ref{thmE}. Indeed, any monoid satisfies Theorem \ref{thmE} (iii), and there exist monoids that are not unions of groups.

\section{Subsemigroups of $\N\times\N$}
\label{sec2}
\setcounter{thm}{0}

This section is devoted to proving Theorem \ref{thmA}, Corollary \ref{corrC} and Theorem \ref{thmB}. We begin by introducing a certain family of subsemigroups of $\N \times\N $ as follows. For $M\subseteq \N $, we let $S_{M}$ be the subsemigroup of $\N\times\N$ generated by the set $1\times M$, i.e.
$$
S_{M} := \bigl\langle (1,m) \::\: m \in M \bigr\rangle \leq \N \times\N .
$$

For a semigroup $S$, we let $SS=\{st\::\: s,t\in S\}$, and call the elements belonging to $S\setminus SS$ \emph{indecomposable}.
Clearly, the indecomposable elements must belong to any generating set of $S$.

\begin{lemma}
\label{lem:irrSM}
The indecomposable elements of $S_M$ are precisely the generators $1\times M$.
\end{lemma}

\begin{proof}
This follows immediately from the fact that $1$ is indecomposable in $\N$.
\end{proof}

Next we record the following criterion for freeness of $S_M$:

\begin{lemma}
\label{nonisomorphicgoodness}
The semigroup $S_M$ is free commutative over $M$ if and only if $|M|\leq 2$.
\end{lemma}

 \begin{proof}
If $|M|=1$, the semigroup $S_{M}$ is clearly free monogenic. If $M=\{m_1,m_2\}$, suppose the generators satisfy a non-trivial relation
 \begin{equation*}
\alpha_{1}(1,m_{1})+ \alpha_{2}(1,m_{2}) = \beta_{1}(1,m_{1})+ \beta_{2}(1,m_{2})
\end{equation*}
(where $\alpha_{1},\alpha_{2},\beta_{1},\beta_{2} \in \N \cup\{0\}$, not all equal to $0$). 
This gives 
\[
\alpha_{1} + \alpha_{2} = \beta_{1} + \beta_{2}, 
\ 
\alpha_{1}m_{1} + \alpha_{2}m_{2} = \beta_{1}m_{1} + \beta_{2}m_{2}.
\]
Denoting $\gamma_{i} = \alpha_{i} - \beta_{i}$ ($i = 1,2$), it follows that
 $\gamma_{1} = -\gamma_{2}$ and $\gamma_{1}m_{1} + \gamma_{2}m_{2} = 0$. 
Since $m_{1} \neq  m_{2}$, we have  $\gamma_{1} = \gamma_{2} = 0$ and thus the relation is trivial. 
Therefore $S_{M}$ is again free.
Finally, if $|M|\geq 3$, pick any three distinct $m_1,m_2,m_3\in M$, and observe that
\begin{equation}
\label{eq:nontriv}
m_{2}(1,m_{1}) + m_{3}(1,m_{2}) + m_{1}(1,m_{3}) = m_{3}(1,m_{1}) + m_{1}(1,m_{2}) + m_{2}(1,m_{3})
\end{equation} 
is a non-trivial relation.
\end{proof}

In the next lemma we show that sets $M$ of size $3$ already yield 
semigroups $S_{M}$ which are typically pairwise non-isomorphic.

\begin{lemma} \label{conditionfornonisomorphism}
 Let $M = \{m_{1},m_{2},m_{3}\}$,  $N = \{n_{1},n_{2},n_{3}\}$ be two 3-element subsets of $\N $. Then $S_{M}$ and $S_{N}$ are isomorphic, via isomorphism $\varphi : S_{M} \rightarrow S_{N}$ satisfying $\varphi(1,m_{i}) = (1,n_{i})$ $(i = 1,2,3)$, if and only if
\begin{equation}
\label{eq:dag}
n_{2}(m_{3}-m_{1}) = n_{1}(m_{3}-m_{2}) + n_{3}(m_{2}-m_{1}).
\end{equation}
\end{lemma}

\begin{proof}
$\left(\Rightarrow\right)$ 
Suppose that $S_{M} \cong S_{N}$ with $\varphi$ the given isomorphism. 
Applying $\varphi$ to the relation \eqref{eq:nontriv}
among the generators $(1,m_1),(1,m_2,),(1,m_3)$ yields
\[
m_{2}(1,n_{1}) + m_{3}(1,n_{2}) + m_{1}(1,n_{3}) = m_{3}(1,n_{1}) + m_{1}(1,n_{2}) + m_{2}(1,n_{3}),
\]
from which \eqref{eq:dag} readily follows.

$\left(\Leftarrow\right)$ 
Assume that \eqref{eq:dag} holds, and define
$\varphi : S_{M} \rightarrow S_{N}$ by
\begin{equation*}
\varphi(x) := \alpha_{1}(1,n_{1}) + \alpha_{2}(1,n_{2}) + \alpha_{3}(1,n_{3})
\end{equation*}
for $x = \alpha_{1}(1,m_{1}) + \alpha_{2}(1,m_{2}) + \alpha_{3}(1,m_{3}) \in S_{M}$. 
We need to show that $\varphi$ is well defined and an isomorphism.
To this end, suppose $\alpha_i,\beta_i\in\N_0$ ($i=1,2,3$), let $\gamma_i=\alpha_i-\beta_i$,
and observe that:
\begin{eqnarray*}
&& \sum_{i=1}^{3}\alpha_{i}(1,m_{i}) =  \sum_{i=1}^{3}\beta_{i}(1,m_{i}) 
 \\
&\Leftrightarrow & \gamma_{1} + \gamma_{2} + \gamma_{3} = 0 \ \& \ \gamma_{1}m_{1} + \gamma_{2}m_{2} + \gamma_{3}m_{3}  = 0 
\\
&\Leftrightarrow & \gamma_{1} = \gamma_{3} \left(\frac{m_{2}-m_{3}}{m_{1}-m_{2}} \right) \ \& \ \gamma_{2} = \gamma_{3}\left(\frac{m_{3}-m_{1}}{m_{1}-m_{2}} \right)
\\
&\Leftrightarrow & \gamma_{1} + \gamma_{2} + \gamma_{3} = 0 \ \& \  \gamma_{1}n_{1} + \gamma_{2}n_{2} + \gamma_{3}n_{3}  = 0 
\\
&\Leftrightarrow &   \sum_{i=1}^{3}\alpha_{i}(1,n_{i}) =  \sum_{i=1}^{3}\beta_{i}(1,n_{i}) 
\\
&\Leftrightarrow &   \varphi\bigl(\sum_{i=1}^{3}\alpha_{i}(1,m_{i})\bigr) =   \varphi\bigl(\sum_{i=1}^{3}\beta_{i}(1,m_{i})\bigr).  
\end{eqnarray*}
It follows that $\varphi$ is well defined and injective.
That it is a homomorphism and surjective follows directly from definition. Therefore $\varphi$ is an isomorphism between $S_{M}$ and $S_{N}$, as required.
\end{proof}

The above characterisation motivates the introduction of the following property. 
We say that a subset $M \subseteq \N $ is \emph{3-separating} if $|M|\geq 3$ and the following condition is satisfied:
\begin{enumerate}[label=\textsf{(S\arabic*)}, widest=(S1), leftmargin=10mm]
\item
\label{it:S1}
 For any two triples $(m_{1},m_{2},m_{3})$ and $(n_{1},n_{2},n_{3})$ of distinct elements from $M$, we have
$$n_{2}(m_{3}-m_{1}) = n_{1}(m_{3}-m_{2}) + n_{3}(m_{2}-m_{1}) \iff (m_{1},m_{2},m_{3}) = (n_{1},n_{2},n_{3}).$$
\end{enumerate}
For example, the set $M = \{1,2,3\}$ is not 3-separating since the triples $(1,2,3)$ and $(3,2,1)$ violate condition \ref{it:S1}. The set $N = \{1,2,4\}$ is 3-separating, which can be verified by direct computation.

Lemma \ref{conditionfornonisomorphism} opens up a way of producing pairs of 3-generator subsemigroups of $\N \times\N $ for which the obvious correspondence between the generators does not induce an isomorphism. We are now going to build an infinite set of generators such that any two distinct subsemigroups generated by 3 elements from the set are actually non-isomorphic. We thus seek to extend our finite examples to an infinite 3-separating set, and we do this inductively. To make the induction work, we introduce an additional condition, and say that a set $M\subseteq \N $ is \emph{strongly 3-separating}, if it is 3-separating and the following condition holds:

\begin{enumerate}[label=\textsf{(S\arabic*)}, widest=(S2), leftmargin=10mm]
\setcounter{enumi}{1}
 \item
\label{it:S2}
For any two pairs $(m_{1},m_{2})$, $(n_{1},n_{2})$ of distinct elements of $M$:$$m_{1} -m_{2} + n_{2} - n_{1} = 0 \iff (m_{1},m_{2}) = (n_{1},n_{2}).$$ 
\end{enumerate}

Noting that subsets of strongly 3-separating sets are again necessarily strongly 3-separating, we show in the next lemma that we can extend a finite strongly 3-separating set to a larger one, whilst maintaining this property.

\begin{lemma}\label{buildingsets}
If $M$ is a strongly 3-separating finite set, then there exists $x \in \N \setminus M$ such that $M\cup\{x\}$ is also strongly 3-separating.
\end{lemma}

\begin{proof}
We consider all the situations where adding an element $x$ to $M$ 
yields a set that is not strongly 3-separating, and prove that there are only finitely many such $x$.

\textit{Case 1: $M \cup \{ x \}$ violates condition \textup{\ref{it:S2}}.} 
This means that there exist two pairs 
$(m_{1},m_{2})\neq (n_{1},n_{2})$
of distinct elements of $M \cup \{ x \}$ 
such that
\begin{equation}
\label{eq:vS2}
m_{1} - m_{2} + n_{2} - n_{1} = 0.
\end{equation}
It follows that at most two of $m_i,n_i$ can equal $x$, and that we cannot have $m_i=n_i=x$.
Thus, if we regard \eqref{eq:vS2} as an equation in $x$, it is linear  and the coefficient of $x$ is non-zero. Thus, for any choice of the $m_i,n_i$ from $M$, of which there are only finitely many, there exists at most one $x$ such that \eqref{eq:vS2} holds.

\textit{Case 2:  $M \cup \{ x \}$ violates condition \textup{\ref{it:S1}}.} 
Now there exist two triples $(m_1,m_2,m_3)\neq(n_1,n_2,n_3)$ of distinct elements of $M\cup\{x\}$
such that
\begin{equation}
\label{eqaaa}
n_{2}(m_{3}-m_{1}) = n_{1}(m_{3}-m_{2}) + n_{3}(m_{2}-m_{1}) .
\end{equation} 
Again, at most one $m_i$ and at most one $n_j$ can equal $x$.

\noindent \textit{Subcase 2.1: Precisely one of $m_1,m_2,m_3,n_1,n_2,n_3$ equals $x$.}
The equation \eqref{eqaaa} is linear in $x$, with the $x$-coefficient $\pm 1$.
Thus, given the five $m_j,n_k\in M$, there is at most one value of $x$ such that
\eqref{eqaaa} holds.

\textit{Subcase 2.2: $x = m_{i} = n_{j}$ for some distinct $i,j \in\{ 1,2,3\}$.}
This time the equation \eqref{eqaaa} is quadratic in $x$, and so there are at most two solutions for $x$ in terms of the remaining four variables $m_{k}, n_{l}\in M$.

\textit{Subcase 2.3: $x = m_{i} = n_{i}$ for some $i = 1,2,3$.}
This time the equation \eqref{eqaaa} is back to linear in $x$, and the coefficient of $x$ has the form $m_{k}-m_{l} + n_{l} - n_{k}$ for $\{k,l\}=  \{1,2,3\}\setminus\{i\}$. 
This coefficient is non-zero because $M$ is assumed to be strongly 3-separating, and in particular it satisfies the condition \ref{it:S2}. So, yet again, there is at most one value of $x$
such that \eqref{eqaaa} holds.
\end{proof}

Iterating Lemma \ref{buildingsets}, and taking the limit we have:

\begin{cor}\label{infinitethreesep}
There exists an infinite strongly \emph{3}-separating set $M_{\infty}$ with $1\in M_\infty$.
\end{cor}

\begin{proof}
Using Lemma \ref{buildingsets}, starting from any finite strongly 3-separating set $M_{1}$
containing $1$ (such as $\{1,2,4\}$), we can build an infinite strictly ascending chain $M_{1} \subset M_{2} \subset M_{3} \subset \dots$ of finite strongly 3-separating sets. Let 
$M_{\infty} := \bigcup_{i \in \N } M_{i}$, and we claim that
$M_{\infty}$ is strongly 3-separating. 
Indeed, if it were not, this would be witnessed by a finite collection of elements (two triples or two pairs), which would be contained in a single $M_i$, and thereby violate the assumption that $M_i$ is strongly 3-separating.
\end{proof}

We can now prove our first main result.

\begin{thm}
There are uncountably many pairwise non-isomorphic subsemigroups of $\N  \times \N $.
\end{thm}

\begin{proof}
Let $M_{\infty}$ be an infinite 3-separating set, whose existence was established in Corollary \ref{infinitethreesep}. Consider the collection of semigroups defined by
$$\mathcal{C} := \bigl\{S_{M} \::\: M \subseteq M_{\infty},\ |M|\geq 3\bigr\}.$$

We claim that the semigroups in $\mathcal{C}$ are pairwise non-isomorphic. Suppose to the contrary that two of these semigroups $S_{M}$ and $S_{N}$ ($M\neq N$) are isomorphic, and let $\varphi$ be an isomorphism between them. 
Without loss of generality assume that $M \setminus N\neq\emptyset$. 
Choose $m_{1} \in M\setminus N$, and let $m_{2},m_{3}$ 
be two further distinct elements of $M$. 

By Lemma \ref{lem:irrSM}, the elements $(1,m_{i})$ are indecomposable in $S_{M}$, and hence their images $\varphi(1,m_{i})$  must be indecomposable in $S_{N}$. 
Again by Lemma \ref{lem:irrSM}, the indecomposables of $S_{N}$ are all of the form $(1,n)$ for $n \in N$, and so there must exist distinct $n_{1},n_{2},n_{3} \in N$ such that $\varphi(1,m_{i}) = (1,n_{i})$ for $i = 1,2,3$. 
It now follows that the subsemigroups 
$\bigl\langle (1,m_{1}), (1,m_{2}), (1,m_{3}) \bigr\rangle\leq S_{M}$ and $\bigl\langle (1,n_{1}), (1,n_{2}), (1,n_{3})\bigr\rangle\leq S_{N}$ are isomorphic via the restriction of $\varphi$.
Now Lemma \ref{conditionfornonisomorphism}, implies that \begin{equation*}n_{2}(m_{3}-m_{1}) = n_{1}(m_{3}-m_{2}) + n_{3}(m_{2}-m_{1}),\end{equation*} which in turn implies that $(m_{1},m_{2},m_{3}) = (n_{1},n_{2},n_{3})$ because $M_{\infty}$
is (strongly) 3-separating. It now follows that $m_1=n_1\in N$, a contradiction with the choice of $m_1$.

This proves that $S_M\not\cong S_N$, and hence $\mathcal{C}$ is indeed an uncountable collection of pairwise non-isomorphic subsemigroups of $\N \times\N $.
\end{proof}

Considering $\N $ as an infinite monogenic subsemigroup, we obtain:

\begin{impcor}If $S$ and $T$ are semigroups containing elements of infinite order, then $S\times T$ contains uncountably many pairwise non-isomorphic subsemigroups. 
\end{impcor}

\begin{proof} If $S$ and $T$ contain elements of infinite order, then they each contain a subsemigroup isomorphic to $\N $. Hence $S\times T$ contains a subsemigroup isomorphic to $\N \times \N $, which contains uncountably many non-isomorphic subsemigroups by Theorem \ref{thmA}. \end{proof}

Our next theorem deals with subdirect products of $\N^k$.

\begin{thm}
For any $k \geq 2$, the direct power $\N ^{k}$ contains uncountably many pairwise non-isomorphic subdirect products.
\end{thm}

\begin{proof}
Let $M_{\infty}$ be an infinite strongly 3-separating set such that $1 \in M_{\infty}$,
guaranteed by Corollary \ref{infinitethreesep}. 
For a subset $M \subseteq M_{\infty}$ containing $1$, define the subsemigroup
\begin{equation*} 
T_{M}:= \bigl\langle(1,\hdots,1, m) \::\: m \in M \bigr\rangle \leq \N ^{k}.
\end{equation*}
Then $T_{M}$ is a subdirect product, as $T_{M}$ contains the diagonal subsemigroup $\bigl\{(n,\dots,n)\::\: n\in\N\bigr\}\leq\N ^{k}$. 
Note that $T_{M} \cong S_{M}$, via isomorphism $\varphi(n,\hdots,n,p) = (n,p)$. The result now follows from Theorem \ref{thmA}.
\end{proof}

\section{Subsemigroups of $\N \times S$ with $S$ finite}
\label{sec3}

In light of Theorem \ref{thmA}, one may ask which directly decomposable semigroups containing $\N $ as a component have only countably many subsemigroups up to isomorphism. 
We have seen that this number is uncountable for $\N\times\N$, while it is trivially finite for $S\times T$ with both $S$ and $T$ finite.
A natural question would be to ask if every finite semigroup $S$ has the property that 
$\N \times S$ contains only countably many subsemigroups. 
We begin by showing that this is at least true for $S$ a finite group.

\begin{lemma}
\label{NxG}
If $G$ is a finite group then every subsemigroup of $\N \times G$ is finitely generated; hence $\N \times G$ has only countably many subsemigroups. 
\end{lemma}

\begin{proof}
Suppose $U \leq \N  \times G$. Since $G$ is finite, there exists $m \in \N $ such that 
$(m,1_{G}) \in U$. 
For every $k \in \N $ define the set
\begin{equation*}
G_{k} := \bigl\{ g \in G \::\: (k,g) \in U\bigr\} \subseteq G.
\end{equation*}

Note that
for $g \in G_{k}$ we have $(k,g) \in U$, and hence 
$(k+m,g)=(k,g)(m,1_{G})  \in U$, implying $g \in G_{k+m}$. 
Hence $G_{k} \subseteq G_{k+m}$ for all  $k \in \N$, and we 
have a chain
\begin{equation*}
G_{k} \subseteq G_{k+m} \subseteq G_{k+2m} \subseteq \hdots .
\end{equation*}

Since $G$ is finite, this chain must eventually stabilise.
It then follows that the sequence $(G_n)_{n\in\N}$ is eventually periodic with period $m$, i.e. 
there exists $t_0\in \N$ such that $G_{t} = G_{t+m}$ for all $t \geq t_{0}$.

 We claim that
 \begin{equation*}
U=\langle X\rangle \text{ where } X: = \bigcup_{1\leq k< t_{0}+m} \bigl(\{k\}\times G_{k}\bigr). 
\end{equation*}
Clearly $\langle X\rangle\subseteq U$, and we just need to show that an arbitrary
$(q,g)\in U$ belongs to $\langle X\rangle$.
We do this by induction on $q$.
For $q< t_0+m$ we have $(q,g)\in X$ and there is nothing to prove.
Suppose now $q\geq t_0+m$.
From $g\in G_q=G_{q-m}$ we have $(q-m,g)\in U$. By induction,
$(q-m,g)\in\langle X\rangle$, and hence
$(q,g)=(q-m,g)(m,1_G)\in \langle X\rangle$, as required.

This proves that an arbitrary subsemigroup of $\N\times G$ is finitely generated.
As there are only countably many finite subsets of the set $\N \times G$, the second assertion follows.
\end{proof}

 Lemma \ref{NxG} is the key observation needed for our next main result, which completely characterises
the direct products $\N\times S$, $S$ finite, with countably many subsemigroups.
In the proof we will make use of \emph{Green's relations} on $S$, which we briefly review;
for a more systematic introduction we refer the reader to \cite{howiefund}. 

Let $1$ be an identity element not belonging to $S$, and let $S^1:=S\cup\{1\}$.
We define three pre-orders and three associated equivalence relations as follows:
\begin{align*}
& s\leq_\R t \Leftrightarrow (\exists u\in S^1)(s=tu), && s\R t\Leftrightarrow s\leq_\R t\ \&\ t\leq_\R s,\\
& s\leq_\L t \Leftrightarrow (\exists u\in S^1)(s=ut), && s\L t\Leftrightarrow s\leq_\L t\ \&\ t\leq_\L s,\\
& s\leq_\J t \Leftrightarrow (\exists u,v\in S^1)(s=utv), && s\J t\Leftrightarrow s\leq_\J t\ \&\ t\leq_\J s.
\end{align*}
Further, we let $\H:=\R\cap\L$. 
If $S$ is finite then $\J=\R\circ\L=\L\circ R=\R\vee\L$ (the composition and join of binary relations).
The maximal subgroups of $S$ are precisely the $\H$-classes of idempotents.
If $S$ is finite and $H$ is a non-group $\H$-class then $h^2<_\J h$ for every $h\in H$.
Thus, a semigroup $S$ is a union of groups (also known as a completely regular semigroup;
see \cite[Section 4.1]{howiefund})
if and only if every $\H$-class contains an idempotent.

\begin{thm} The following are equivalent for a finite semigroup $S$:
\begin{enumerate}[label=\textup{(\roman*)}, widest=(iii), leftmargin=10mm]
\item 
$\N \times S$ has only countably many subsemigroups;
\item 
$\N \times S$ has only countably many pairwise non-isomorphic subsemigroups;
\item 
$S$ is a union of groups.
\end{enumerate}
\end{thm}

\begin{proof}
The implication \ref{it:D1}$\Rightarrow$\ref{it:D2} is immediate.

\ref{it:D2}$\Rightarrow$\ref{it:D3} 
We prove the contrapositive: if $S$ is not a union of groups then $\N \times S$ has uncountably many pairwise non-isomorphic subsemigroups. Note that $S$ not being a union of groups means that there exists a non-group $\mathcal{H}$-class $H$ of $S$.
Let $x \in H$. 
From finiteness of $S$ we have $x^{2k} = x^{k}$ for some $k \in \mathbb{N}$. 
Since $H$ is non-group we have $x^2<_\J x$, and,
more generally, $x^i<_\J x$, so that we must have $k>1$, and there can be no
$y \in S$ such that $yx^{k} = x$ or $x^{k}y = x$. 

For any $M \subseteq \N \setminus\{1\}$ define the semigroup
\begin{equation*}
S_{M} := \bigl\langle(1,x^{k}), (m,x): m \in M \bigr\rangle \leq \N \times S.
\end{equation*}
Since $1$ is indecomposable in $\N$ and $x$ is indecomposable in $\langle x\rangle\leq S$, it follows that 
all the generators of $S_M$ are indecomposable.

We claim that the semigroups $S_{M}$ are pairwise non-isomorphic. Suppose to the contrary that $S_{M} \cong S_{N}$ for some $M \neq N$ via isomorphism $\varphi: S_{M} \rightarrow S_{N}$. Without loss assume that there exists $m\in M\setminus N$.
Since $x^k$ is an idempotent, we have
$(1,x^k)^{mk}=(m,x)^k$.
Applying $\pi_1\varphi$, where $\pi_1$ stands for the projection to the first component $\N$, yields
\begin{equation}
\label{samefirstcoord}
m\cdot \pi_{1}\varphi(1,x^{k}) = \pi_{1}\varphi(m,x).
\end{equation}
Recalling that $m\neq 1$, this implies
$\pi_{1}\varphi(1,x^{k}) <  \pi_{1}\varphi(m,x)$,
and hence $\varphi(m,x)\neq (1,x^k)$.
Since $(1,x^k)$ is indecomposable in $S_N$ it follows that we must have
$\varphi(1,x^k)=(1,x^k)$. But then \eqref{samefirstcoord} yields
$m=\pi_1\varphi(m,x)\in S_N$, a contradiction.
Thus $M=N$, and $\bigl\{ S_M\::\: M\subseteq \N\setminus\{1\}\bigr\} $ is indeed an uncountable collection of pairwise non-isomorphic subsemigroups of $\N\times S$.

\ref{it:D3}$\Rightarrow$\ref{it:D1}
Suppose that $S$ is a union of groups,
i.e. that every $\H$-class $H_x$ ($x\in S$) is a group. From
\begin{equation*}
\N \times S =  \N \times \bigl( \bigcup_{x \in S} H_{x} \bigr) = \bigcup_{x \in S}\left( \N \times H_{x}\right),
\end{equation*}
we see that $\N \times S$ is a finite (disjoint) union of semigroups $\N\times H_x$, each 
of which has only countably many subsemigroups by Lemma \ref{NxG}. 
It follows that $\N \times S$ itself has only countably many subsemigroups.
\end{proof}

Turning to subdirect products, we have our final main result:

\begin{thm} The following are equivalent for a finite semigroup $S$:
\begin{enumerate}[label=\textup{(\roman*)}, widest=(iii), leftmargin=10mm]
\item 
$\N \times S$ has only countably many subdirect products;
\item 
$\N \times S$ has only countably many pairwise non-isomorphic subdirect products;
\item 
For every $s \in S$, there exists some $t \in S$ such that at least one of $ts = s$ or $st = s$ holds.
\end{enumerate}
\end{thm}

\begin{proof}
The implication \ref{it:E1}$\Rightarrow$\ref{it:E2} is immediate.

\ref{it:E2}$\Rightarrow$\ref{it:E3} 
We prove the contrapositive:
if there exists $s \in S$ such that 
\begin{equation}
\label{eq:sass}
st\neq s\text{ and } ts\neq s \text{ for all } t\in S,
\end{equation}
then $\N \times S$ has uncountably many non-isomorphic subdirect products. 

We begin by claiming that \eqref{eq:sass} also implies
\begin{equation}
\label{eq:sprop}
ust \neq  s \text{ for all } u,t\in S.
\end{equation}
For, otherwise, if $s,t \in S$ were such that $ust = s$ for some $u,t\in S$,
then  we would have $u^{n}st^{n} = s$ for all $n \in \N $. Letting $j \in \N $ be such that $u^{j} = u^{2j}$, we see that
$s = u^{2j}st^{2j} = u^{j}st^{2j} = st^{j}$, contradicting \eqref{eq:sass}.

As $S$ is finite, we have $s^k=s^{2k}$ for some $k$,
and $k>1$ by \eqref{eq:sass}.
For $M \subseteq \N \setminus\bigl(2\N  \cup \{1\}\bigr)$, let
 \begin{equation*}
S_{M} := \bigl\langle(1,s^{k}),(2,t), (m,s): t \in S\setminus\{s,s^{k}\},\ m \in M \bigr\rangle \leq \N \times S.
\end{equation*}
Clearly, $S_M$ is a subdirect product, since $1$ belongs to the first projection of the generating set, while its second projection already contains the entire semigroup $S$.

Next we claim that all the generators are indecomposable in $S_{M}$. 
This is clear for $(1,s^k)$.
The only decomposable element in $S_{M}$ of the form $(2,t)$ is $(1,s^{k})^{2} = (2,s^{k})$, which is explicitly excluded from the generators. 
Finally, a generator of the form $(m,s)$ cannot be expressed as a non-trivial product of generators, as such a product cannot be just a power of $(2,t)$ because $m$ is odd,
and it cannot include a generator $(1,s^k)$ or $(m^\prime,s)$ because of 
\eqref{eq:sass}, \eqref{eq:sprop}.

Let $m\in M$ be arbitrary.
In the same way as in the proof of Theorem \ref{thmD}, applying $\pi_1\varphi$ to
$(1,s^k)^{mk}=(m,s)^k$, yields
\begin{equation}
\label{samefirstcoord2} 
m\cdot\pi_{1}\varphi(1,s^{k})=  \pi_{1}\varphi(m,s) \text{ for all } m\in M,
\end{equation}
and, as a consequence,
\begin{equation}
\label{geqcoords}
\pi_{1}\varphi(1,s^{k}) <  \pi_{1}\varphi(m,s) \text{ for all } m\in M.
\end{equation}

We now claim that if $M \neq N$, then $S_{M} \not \cong S_{N}$. 
Suppose to the contrary that $S_{M} \cong S_{N}$ via isomorphism $\varphi : S_{M} \rightarrow S_{N}$.

We claim that
\begin{equation}
\label{eq:pi1}
\pi_1\varphi(1,s^k)=1.
\end{equation}

Since $(1,s^{k})$ is indecomposable in $S_{M}$, the element $\varphi(1,s^{k})$ must be indecomposable in $S_{N}$, hence $\pi_{1}\varphi(1,s^{k}) \in \{1,2\} \cup N$.

Suppose that $\pi_{1}\varphi(1,s^{k}) = 2$. 
Let $m\in M$ be arbitrary.
Then $\pi_{1}\varphi(m,s) = 2m$ 
by \eqref{samefirstcoord2}. 
But $\varphi(m,s)$ is indecomposable in $S_{N}$, hence $\pi_{1}\varphi(m,s) \in \{1,2\}\cup N$. As $N$ was chosen to consist of only odd numbers, it follows that 
$\pi_{1}\varphi(m,s) = 2$, which implies $1=m\in M$, a contradiction with the choice of $M$. 

Now suppose that $2<\pi_{1}\varphi(1,s^{k})\in N$. 
Then, by (\ref{geqcoords}), it follows that $\pi_{1}\varphi(m,s) > 2$ for every $m \in M$. 
In other words
\[
\varphi\Bigl(\bigl\{(1,s^k)\bigr\}\cup\bigl\{(m,s)\::\: m\in M\bigr\}\Bigr)\subseteq \bigl\{ (n,s)\::\: n\in N\bigr\}.
\]
Since the generators of both $S_M$ and $S_N$ are indecomposable, and since $\varphi$ is an isomorphism, we would have to have
\[
\varphi\Bigl(\bigl\{ (2,t)\::\: t\in S\setminus \{s,s^k\}\bigr\}\Bigr)\supseteq \bigl\{(1,s^k)\bigr\}\cup 
\bigl\{ (2,t)\::\: t\in S\setminus \{s,s^k\}\bigr\},
\]
which is clearly impossible on account of their sizes.
This completes the proof of \eqref{eq:pi1}.

Now assume that $M\neq N$, and, without loss, that there exists $m\in M\setminus N$.
By \eqref{samefirstcoord2}, $m=\pi_{1}\varphi(m,x) \in N$, a contradiction. 
It follows that
$\bigl\{S_{M} \::\: M \subseteq \N \setminus(2\N \cup\{1\}) \bigr\}$
is an uncountable collection of pairwise non-isomorphic subdirect products of $\N \times S$.

\ref{it:E3}$\Rightarrow$\ref{it:E1} 
We will prove that every subdirect product
$T \leq \N \times S$ is finitely generated, and the assertion will follow.

For every $n\in\N$ consider the set
\begin{equation*} 
S_{n} = \{s \in S : (n,s) \in T\}.
\end{equation*}
As $T$ is subdirect, for every $s \in S$ we can choose $m_{s} \in \N $ such that $(m_{s},s) \in T$. 
Let $m$ be the least common multiple of all of the $m_{s}$. 

We claim that
\begin{equation}
\label{eq:Sninc}
S_n \subseteq S_{n+m} \text{ for all } n\in \N.
\end{equation}
Indeed, suppose $s\in S_n$, so that $(n,s)\in T$.
By assumption, there exists $t\in S$ such that $st=s$ or $ts=s$; without loss assume $st=s$.
Then, writing $m=lm_t$ for some $l \in N$, we have
$(n+m,s)=(n,s)(m_t,t)^l\in T$, and hence $s\in S_{n+m}$, as required.

Thus, for every $n\in\N$, we have an infinite chain
$S_n\subseteq S_{n+m}\subseteq S_{n+2m}\subseteq\dots$, which must eventually stabilise
because $S$ is finite.
It follows then that the entire sequence $(S_n)_{n\in \mathbb{N}}$ is eventually periodic
with period $m$, i.e.
there exists $n_0\in\N$ such that
\begin{equation}
\label{eq:Sneq}
S_n=S_{n+m} \text{ for all } n\geq n_0.
\end{equation}

We now claim that
\begin{equation}
\label{eq:Tgen}
T=\langle X\rangle \text{ where }
X:= \bigcup_{1\leq n<n_0+m} \{n\}\times S_n .
\end{equation}
It is clear that $\langle X\rangle\subseteq T$.
To prove the converse inclusion, consider an arbitrary $(q,s)\in T$.
We prove by induction on $q$ that $(q,s)\in\langle X\rangle$.

If $q<n_0+m$ the element already belongs to $X$, and there is nothing to prove.
So suppose $q\geq n_0+m$. Then $q-m\geq n_0$, and hence
$S_{q-m}=S_q$ by \eqref{eq:Sneq}.
It now follows that $(q-m,s)\in T$, so, by induction, $(q-m,s)\in\langle X\rangle$.
By assumption, there exists $t\in S$ with $st=s$ or $ts=s$; without loss assume the former is the case.
Write $m=lm_t$, recall that $(m_t,t)\in X$, and then we have
$(q,s)=(q-m,s)(m_t,t)^l$.
Noting that $(m_t,t)\in X$, because $m_t\leq m$, we conclude that $(q,s)\in\langle X\rangle$,
completing the proof of finite generation of $T$, and hence of the theorem.
\end{proof}

\section{Some further questions}

Combinatorial properties of subdirect products of semigroups have so far been somewhat neglected in literature.
We believe that they offer fertile ground for future research.
By way of encouraging such work, we offer a few questions which seem to naturally offer themselves following the results of this paper.

\begin{que}
Is it possible to characterise all pairs of finitely generated commutative semigroups $S$, $T$ such that
there are only countably many pairwise non-isomorphic subdirect products of $S$ and $T$?
\end{que}

We remind the reader that finitely generated commutative semigroups are finitely presented (see \cite[Section VI.1]{grillet}) and hence there are only countably many possible choices for $S$ and $T$.

\begin{que}
Given a fixed finitely generated infinite commutative semigroup $S$, is it possible to characterise all finite semigroups $T$ such that $S\times T$ has only countably many pairwise non-isomorphic subsemigroups or subdirect products? Do these characterisations depend on $S$?
\end{que}

\begin{que}
How many pairwise non-isomorphic subsemigroups and subdirect products does $F\times F$ contain, where $F$ is a finitely generated free semigroup in some other well known semigroup varieties, such as inverse semigroups or completely regular semigroups?
\end{que}

In group theory, subdirect products of several factors which project (virtually) onto any \emph{pair} of factors turn out to be easier to handle than the general subdirect products;
see for example \cite{bridson09,bridson13}. In \cite{pmnr} it is shown that the situation is likely to be more complicated for semigroups. In the context of subdirect products of copies of $\N$ we ask:

\begin{que}
Is it true that for every $k\in \N$ there are uncountably many pairwise non-isomorphic subdirect products of $\N^k$ which project onto any $k-1$ factors?
\end{que}

\end{document}